\documentclass[12pt,centertags,oneside]{amsart}
\RequirePackage[utf8]{inputenc}
\usepackage{amstext,amsthm,amscd,typearea,hyperref}
\usepackage{amsmath}
\usepackage{amssymb}
\usepackage{a4wide}
\usepackage[mathscr]{eucal}
\usepackage{mathrsfs}
\usepackage{typearea}
\usepackage{charter}
\usepackage{pdfsync}
\usepackage[a4paper,width=16.2cm,top=3cm,bottom=3cm]{geometry}

\usepackage{multicol}

\numberwithin{equation}{section}

\newtheorem{theorem}{Theorem}[section]

\newtheorem{proposition}[theorem]{Proposition}
\newtheorem{corollary}[theorem]{Corollary}
\newtheorem{lemma}[theorem]{Lemma}
\newtheorem{remark}[theorem]{Remark}
\newtheorem{example}[theorem]{Example}

\newtheorem{prop}[theorem]{Proposition}

\newtheorem{defn}[theorem]{Definition}

\newcommand{\diff}{{\rm d}}

\newcommand{\del}{\partial}

\newcommand{\dist}{\mathop{\mathrm{dist}}\nolimits}

\newcommand{\ddc}{dd^c}

\newcommand{\id}{{\rm id}}

\newcommand{\DD}{\mathbb{D}}

\renewcommand\P{\mathbb{P}}

\newcommand{\lp}{\langle}
\newcommand{\rp}{\rangle}

\title{Dynamics of holomorphic correspondences on Riemann Surfaces}

\author{Tien-Cuong Dinh, Lucas Kaufmann and Hao Wu}
\address{Department of Mathematics,  National University of Singapore - 10, Lower Kent Ridge Road - Singapore 119076}
\email{matdtc@nus.edu.sg; lucaskaufmann@nus.edu.sg; e0011551@u.nus.edu}

\date{}

\begin{document}

\begin{abstract}
We study the dynamics of holomorphic correspondences $f$ on a compact Riemann surface $X$ in the case, so far not well understood, where $f$ and $f^{-1}$ have the same topological degree. Under a mild and necessary condition that we call non weak modularity,  $f$ admits two canonical probability measures $\mu^+$ and $\mu^-$ which are invariant by $f^*$ and $f_*$ respectively. If the critical values of $f$ (resp. $f^{-1}$) are not periodic, the backward 
(resp. forward) orbit of any point $a \in X$ equidistributes towards $\mu^+$ (resp. $\mu^-$), uniformly in $a$ and exponentially fast.
\end{abstract}

\maketitle

%

\section{Introduction and Main results}

    Let $X$ be a compact Riemann surface. Denote by $\pi_1$ and $\pi_2$ the canonical projections from $X\times X$ to its factors. A \textit{holomorphic correspondence} on $X$ is an effective analytic cycle $\Gamma = \sum_i \Gamma_i$ in $X\times X$ of pure dimension one containing no fiber of $\pi_1$ or $\pi_2$. Here, $\Gamma_i$ are irreducible but not necessarily distinct. A correspondence induces an action on subsets $A$ of $X$, which we denote by $f$, given by the following rule $$f(A):=\pi_2(\pi_1^{-1}(A)\cap\Gamma) \quad \text{and} \quad f^{-1}(A):=\pi_1(\pi_2^{-1}(A)\cap\Gamma),$$ where we can count points with multiplicity. We call $\Gamma$ {\it the graph} of $f$.
    
The correspondence $f$ can be viewed as a multi-valued map from $X$ to itself, where the value of $f$ at $x \in X$  is the finite set $f(x)$. The adjoint correspondence $f^{-1}$ of $f$ is the correspondence whose graph is the image of $\Gamma$ by the involution $(x,y) \mapsto (y,x)$. In general, we don't have $f\circ f^{-1}=\id$ as in the case of maps.

    \vskip5pt
Denote by $d_1(f)$ and $d_2(f)$ the degree of $\pi_1|_\Gamma$ and $\pi_2|_\Gamma$ respectively. The number $d_2(f)$ is called the \textit{topological degree} of $f$. Notice that $d_1(f) = d_2(f^{-1})$ and $d_2(f) = d_1(f^{-1})$. If $g$ is another correspondence we can define the composition $f \circ g$ (see Section \ref{sec:preliminaries}) and we have $d_j(f \circ g) = d_j(f) \cdot d_j(g)$, $j=1,2$. In particular we can define the $n^{th}$ iterate $f^n$ of $f$ and study its dynamics.
    
    When $d_2(f) > d_1(f)$, the global dynamics of $f$ is more or less  well understood and the situation is similar to the case of rational functions of degree larger than $1$, see \cite{dinh:correspondance-polynomiale,dinh-sibony:distribution} and also the surveys \cite{fornaess:survey,sibony:panorama,dinh-sibony:cime}. In that case, there is a canonical probability measure $\mu$ on $X$ given by 
    $$\mu = \lim_{n \to \infty} \frac{1}{d_2(f)^n} (f^n)^* \omega,$$
where $\omega$ is a fixed volume form of integral one (see Section \ref{sec:preliminaries} for the definition of the pullback operator). The measure $\mu$ is invariant in the sense that $f^* \mu = d_2(f) \cdot \mu$ and, among many other good properties, it describes the distribution of repelling periodic points and that of pre-images of a generic point. We refer to \cite{dinh:correspondance-polynomiale} for more details. The support of $\mu$ is disjoint from the normality set of $f$ (an analogue of the Fatou set of a rational function), see \cite{bharali-sridharan}. When $d_2(f) < d_1(f)$ we can apply the above results for the adjoint correspondence $f^{-1}$.
    
   The remaining case when $d_1(f) = d_2(f)$, that is, when $f$ and $f^{-1}$ have the same topological degree, is more delicate and the known methods to construct $\mu$ and study its properties do not apply directly. In the special situation of modular correspondences, $X$ is already equipped with a natural measure which is invariant by the dynamics (see  Mok \cite{mok:correspondences}, Clozel-Ullmo \cite{clozel-ullmo} and Example \ref{ex:modular-correspondence} below) and some equidistribution theorems can be obtained, see Clozel-Otal \cite{clozel-otal} and \cite{dinh:modular}.
   
   Our first main result gives the existence of two canonical measures for a correspondence $f$ with $d_1(f) = d_2(f)$ satisfying a condition of non-modularity. A correspondence $f$ with graph $\Gamma$ is called \textbf{weakly modular} if there exists a  positive measure $m$ on $\Gamma$ and probability measures $m_1$ and $m_2$ on $X$ such that 
$ m=(\pi_1|_\Gamma)^*(m_1)$ and  $m=(\pi_2|_\Gamma)^*(m_2)$. 

\begin{theorem}\label{thm:main-theorem}
		Let $f$ be a non-weakly modular correspondence on a compact Riemann surface $X$ such that $d_1(f) = d_2(f) = d \geq 2$. There exist two probability measures $\mu^+, \mu^-$ on $X$ which are invariant by $f^*, f_*$  in the sense that  
		$$f^* \mu^+ = d \, \mu^+ \qquad \text{and} \qquad f_* \mu^- = d \, \mu^-$$ 
		and such that if $\alpha$ is any smooth $(1,1)$-form on $X$ and $c_\alpha:=\int_X\alpha$,
		then
		\begin{equation} \label{eq:def-mu-}
		\frac{1}{d^n}(f^n)^*\alpha \to c_\alpha\mu^+ \qquad \text{and} \qquad 
		\frac1{d^n}(f^n)_*\alpha \to c_\alpha \mu^-
		\qquad \text{as} \qquad n \to \infty.
		\end{equation}
\end{theorem}

We will see in the proof that the above convergences are exponentially fast. This is a strong mixing property of the system, see also Proposition \ref{prop:mixing} below. Note also that being weakly-modular is a very restrictive property and the above theorem applies to a wide class of correspondences. 
Moreover, the theorem is no longer true if we remove the non-weak modularity condition. A surprising feature of the above result is that, in contrast to the case where $d_1(f) \neq d_2(f)$, we have two canonical measures which share the role of the global description of the system. In general they are different but in some cases, e.g.\ when $f=f^{-1}$, they can be equal.
  
The known methods used to produce a canonical invariant measure in the case $d_1(f) \neq d_2(f)$ rely in some way or another on the spectral gap of the action of $\frac{1}{d}f_*$ or $\frac{1}{d}f^*$ in cohomology. In our case, such a spectral gap doesn't exist. Surprisingly enough, in the absence of a modular measure as above, we can prove that the actions of  $\frac{1}{d}f_*$ and $\frac{1}{d}f^*$  on the Sobolev space $W^{1,2}$ of $X$ have good spectral properties, which allow us to adapt some of those methods to our setting.

Once we have canonical measures it is natural to ask if they describe the distribution of images and pre-images of a given point by $f^n$. The following result gives an answer to this question, showing that under a  condition on $f$, the pre-images (resp.\ images) of any point equidistribute exponentially fast towards $\mu^+$ (resp.\ $\mu^-$). 

\begin{theorem} \label{thm:equid-no-periodic-crit}
Let $f, X, d, \mu^+$ and $\mu^-$ be as in Theorem \ref{thm:main-theorem}. Suppose moreover that no critical value of $f$ is periodic. Then there is a constant $0<\lambda_+ < 1$ such that for any $a \in X$ and every test function $\varphi$ of class $\mathcal C^\beta$ on $X$, with $0<\beta \leq 1$,  we have
\begin{equation} \label{eq:equid-no-periodic-crit}
\big| \big \langle d^{-n}(f^n)^*\delta_a - \mu^+, \varphi \big \rangle \big | \leq A^+_\beta \|\varphi\|_{\mathcal C^\beta} \lambda_+^{\beta n} \quad \text{for every} \quad n \geq 0,
\end{equation}
where $A^+_\beta > 0$ is a constant independent of $n$, $a$ and $\varphi$.
Analogously, if  no critical value of $f^{-1}$ is periodic, then there is a constant $0<\lambda_- <1$ such that for any $a \in X$ and every test function $\varphi$ of class $\mathcal C^\beta$ on $X$, with $0<\beta \leq 1$, we have
\begin{equation} \label{eq:equid-no-periodic-crit2}
\big| \big \langle d^{-n}(f^n)_*\delta_a - \mu^-, \varphi \big \rangle \big | \leq A^-_\beta \|\varphi\|_{\mathcal C^\beta} \lambda_-^{\beta n} \quad \text{for every} \quad n \geq 0,
\end{equation}
for some constant $A^-_\beta > 0$  independent of $n$, $a$ and $\varphi$.
\end{theorem}

\vskip5pt
The condition on the critical values of $f$ or $f^{-1}$ in the above theorem is needed to get the convergence rate. It is natural to ask in which generality such a result holds. In the general case, it is interesting to understand the exceptional set of $f$, i.e., the maximal finite set $E \subset X$ such that $f^{-1}(E) \subset E$.  We expect convergence, without speed, outside the orbit of the exceptional set (which may be dense in $X$), but the situation seems to be more delicate then the one of maps or correspondences with  $d_1(f) \neq d_2(f)$. The main problem is that the size of the orbit of the critical values increases faster than the one with $d_1(f) \neq d_2(f)$.  For this reason, the geometrical methods of Freire-Lopes-Ma{\~n}{\'e} \cite{freire-lopes-mane} and Lyubich \cite{lyubich} do not apply. For Theorem \ref{thm:equid-no-periodic-crit}, we will use an analytic method which is related to the one introduced by Sibony and the first author in \cite{dinh-sibony:equid-speed}. 

Another natural question is whether $\mu^{\pm}$ describe the distribution of periodic points. We postpone these questions to a later work. 
We expect that about half of periodic the points are repelling and equidistributed with respect to $\mu^+$ and about half of them are contracting and equidistributed with respect to $\mu^-$, see Proposition \ref{p:graph} below. This indicates that the system seems to be equally expansive and contractive. This is a surprising new phenomenon.

Our methods can also be used to study the dynamics of a collection of automorphisms. An interesting case consists of considering the action of a collection of M\"obius transformations acting on $\P^1$, e.g.\ ,generators of a Fuschian or Kleinian group. Some results in this direction shall appear in a forthcoming paper.
\\

\textbf{Acknowledgements:} This work was supported by the NUS grants C-146-000-047-001,  AcRF Tier 1 R-146-000-248-114  and R-146-000-259-114.

\section{Preliminary results} \label{sec:preliminaries}

Let $X$ be a compact Riemann surface. As seen in the introduction, a correspondence $f$ on $X$ is given by its graph $\Gamma = \Gamma_f \subset X \times X$ and we denote by $d_i(f)$, $i=1,2$, the degree of $\pi_i|_\Gamma$. We sometimes identify $\Gamma$ with its support and we often assume that  $d_1(f) = d_2(f) = d \geq 2$.

\subsection*{Composition and pullback} If $f$ and $g$ are two correspondences on $X$ we define $f \circ g$ in the following way. Suppose first that $\Gamma_f$ and $\Gamma_g$ are irreducible subvarieties. The product $\Gamma_g \times \Gamma_f$ is naturally included in $X^4 = \{(x_1,x_2,x_3,x_4) : x_i \in X\}$. Define $\widehat \Gamma_{f \circ g} = (\Gamma_g \times \Gamma_f) \cap \{x_2=x_3\}$ and let $ \Gamma_{f \circ g}$ be the push forward of $\widehat \Gamma_{f \circ g}$ in the sense of cycles by the projection $(x_1,x_2,x_3,x_4) \mapsto (x_1,x_4)$. Then $f \circ g$ is by definition the correspondence whose graph is $ \Gamma_{f \circ g}$. If $z \in f \circ g(x)$, the pushforward takes into account the number of $y$'s such that $y \in g(x)$ and $z \in f(y)$.  In general, we write $\Gamma_f = \sum_i \Gamma_f^i$ and $\Gamma_g = \sum_j \Gamma_g^j$ where $\Gamma^i_f$ and $\Gamma^j_g$ are irreducible and extend the above definition by linearity.

The composition is associative and $d_j(f \circ g) = d_j(f) \cdot d_j(g)$, $j=1,2$. In particular we can consider iterates of $f$ and we have $d_j(f^n) = d_j(f)^n$ for every $n\geq 1$.

A correspondence induces a push-forward operator and a pullback operator on currents. When $S$ is a smooth form, a continuous function or a finite measure, we have
$$ f_*(S):=(\pi_2)_*(\pi_1^*(S)\wedge [\Gamma]) \quad \text{and} \quad f^*(S):=(\pi_1)_*(\pi_2^*(S)\wedge[\Gamma]).$$
The operators $f^*$ and $f_*$ preserve the degree of currents and when $S$ is a smooth form, $f^*(S)$ and $f_*(S)$ will be smooth forms outside some finite sets related to ramification points (see below for more details). If $\varphi$ is a continuous function we have $(f_*\varphi)(y) = \sum_{x \in f^{-1}(y)} \varphi(x)$ where the points are counted with multiplicity. One can easily see that if $\mu$ is a probability measure on $X$ then $f_* (\mu)$ and $f^*(\mu)$ are positive measures whose total masses are $d_1(f)$ and $d_2(f)$ respectively.

\subsection*{Ramification locus and local branches} From now on, assume that $d_1(f) = d_2(f) = d \geq 2$. 
A  \textit{branch} of $f$ over an open subset $U$ of $X$ is a biholomorphic map $\tau: U \to V \subset X$ whose graph is contained in $\Gamma$. We say that two branches are different if their graphs are different. 
When the components of $\Gamma$ are not distinct, the branches of $f$ need to be counted with multiplicity. 
An \textit{inverse branch} of $f$ is a branch of $f^{-1}$.

Denote by $R_i$, $i=1,2$, the finite set of points $a\in\Gamma$ such that in any neighbourhood $W$ of $a$, $\pi_i$ is not injective on at least 
 one component of $\Gamma\cap W$.  Define $B_i := \pi_i(R_i)$. 
We say that $B_2$ (resp. $B_1$) is {\it the set of critical values} for $f$ (resp. for $f^{-1}$) and $R_2$ (resp. $R_1$) is {\it the set of ramification points} of $f$ (resp. $f^{-1}$). 
Note that on any simply connected open subset of $X\setminus B_2$, $f$ admits  exactly $d$ inverse branches, counting multiplicity.
A similar property holds for  $X\setminus B_1$ and $f^{-1}$.

\subsection*{Action on $L^2$ forms and a Cauchy--Schwarz type inequality} Consider the space 
$$L^2_{(1,0)} := \big\{\phi : \phi \text{ is  a } (1,0) \text{-form on } X \text{ with } L^2 \text{ coefficients} \big\}.$$
Observe that if  $\phi \in L^2_{(1,0)}$, then  $i \phi \wedge \overline \phi$ is a positive $(1,1)$-form with $L^1$ coefficients. Therefore, we can equip $L^2_{(1,0)}$ with the norm
\begin{equation} \label{eq:L2norm}
\|\phi\|_{L^2} = \Big( \int_X i \phi \wedge \overline \phi \Big)^{1/2}.
\end{equation}
We define $L^2_{(0,1)}$ and its norm in the same way.

The following simple result will be crucial for us.

\begin{proposition} \label{prop:CS-ineq}
Let $f$ be as above with $d_1(f)=d_2(f)=d$. Then the operator $f^*$, which is well-defined on smooth $(1,0)$-forms (resp. $(0,1)$-forms), extends continuously to a linear bounded operator $f^*:L^2_{(1,0)}\to L^2_{(1,0)}$ (resp. $f^*:L^2_{(0,1)}\to L^2_{(0,1)}$) with norm bounded by $d$, that is, for $\phi$ in $L^2_{(1,0)}$ (resp. $L^2_{(0,1)}$), we have the inequality $\|f^*\phi\|_{L^2}\leq d\|\phi\|_{L^2}$. Moreover, the equality holds 
if and only if for any open set $ U \subset X$ avoiding the critical values of $f^{-1}$ and any two local branches $\tau_1$, $\tau_2$ of $f$ defined in $U$, we have $\tau_1^* \phi = \tau_2^* \phi$ on $U$. 
\end{proposition}
\begin{proof}
We only consider the case of bi-degree $(1,0)$ since the case of bi-degree $(0,1)$ can be treated in the same way. 
Note that the last property in the proposition is local on $X\setminus B_1$. Hence, it is enough to consider small open subsets $U$ of $X\setminus B_1$ which are simply connected.  Thus, $f$ has $d$ branches $\tau_j: U \to V_j$ over $U$ counting the multiplicity.

Consider first the case where $\phi$ is smooth. 
We have $(f^*\phi)|_U = \sum_{j=1}^d \tau_j^*\phi$. So on the open set $U$ we have
\begin{eqnarray*}
i f^* \phi \wedge \overline {f^* \phi} &= & i \Big(\sum_{1\leq j \leq d} \tau_j^* \phi \Big) \wedge \Big(\sum_{1\leq j \leq d} \tau_j^* \overline \phi \Big)  \\ 
&=& d  \sum_{1\leq j \leq d} i \tau_j^* \phi \wedge \overline{\tau_j^* \phi} - \sum_{1\leq j<l\leq d} i (\tau^*_j\phi - \tau^*_l \phi) \wedge \overline{(\tau^*_j \phi- \tau^*_l  \phi)} \\ 
& = & d \, f^*(i \phi \wedge \overline \phi) - \sum_{1\leq j<l\leq d} i (\tau^*_j\phi - \tau^*_l \phi) \wedge \overline{(\tau^*_j \phi- \tau^*_l  \phi)} \\ 
& \leq & d \, f^*(i \phi \wedge \overline \phi)
\end{eqnarray*}
with the equality holding if and only if $\tau_j^* \phi = \tau_l^* \phi$ on $U$ for every $j,l$.

From the fact that finite sets have zero Lebesgue measure and that for any positive measure $m$ the mass of the measure $f^* m$ is $d$ times the mass of $m$, we get 
$$\int_X i  f^* \phi \wedge \overline{f^* \phi} \leq d  \int_X f^*(i \phi \wedge \overline \phi) = d^2 \int_X i \phi \wedge \overline \phi,$$
which gives $\|f^* \phi\|_{L^2} \leq d\|\phi\|_{L^2}$. This is the desired estimate for $\phi$ smooth.

Now, since smooth forms are dense in $L^2_{(1,0)}$, the operator $f^*$ extends continuously to an operator on $L^2_{(1,0)}$ with norm bounded by $d$. By continuity, the identity $(f^*\phi)|_U = \sum_{j=1}^d \tau_j^*\phi$ still holds for $\phi\in L^2_{(1,0)}$ and $f^*\phi$ has no mass on finite sets as it is of class $L^2$. We conclude that the above computation still holds for $\phi\in L^2_{(1,0)}$ and we have the equality $\|f^* \phi\|_{L^2} = d\|\phi\|_{L^2}$ if and only if $\tau_j^* \phi = \tau_l^* \phi$ on $U$ for every $j,l$.
\end{proof}

\begin{remark} \rm
Since $f_*=(f^{-1})^*$, Proposition \ref{prop:CS-ineq} also holds for $f_*$ instead of $f^*$. In this case, we need to consider open subsets $U$ of $X\setminus B_2$ and inverse branches of $f$ on $U$. 
\end{remark}

The following lemma will be used later.

\begin{lemma} \label{l:L2-branches}
Let $\phi$ be a form in $L^2_{(1,0)}$  such that $\|\phi\|_{L^2}=1$. Define 
$$\nu_\phi:= df^*(i \phi \wedge \overline \phi) - i  f^* \phi \wedge \overline{f^* \phi}.$$
Let $\tau_1,\tau_2,\tau_3$ and $\tau_4$ be arbitrary branches of $f$ on some open set $U$. Then we have the following mass estimate
$$\|\tau_1^*\phi\wedge \overline{\tau_2^*\phi} -  \tau_3^*\phi\wedge \overline{\tau_4^*\phi}\| \leq 2d^{1/2} \|\nu_\phi\|^{1/2}.$$
\end{lemma}
\proof
We can write $\tau_1^*\phi\wedge \overline{\tau_2^*\phi} -  \tau_3^*\phi\wedge \overline{\tau_4^*\phi}$ as 
\begin{equation} \label{e:CS-branch}
\tau_1^*\phi\wedge \overline{(\tau_2^*\phi -\tau_4^*\phi)}  +  (\tau_1^*\phi-\tau_3^*\phi)\wedge \overline{\tau_4^*\phi}.
\end{equation}
We estimate the first term of the last sum. By Cauchy-Schwarz's inequality, we have
$$\|\tau_1^*\phi\wedge \overline{(\tau_2^*\phi -\tau_4^*\phi)}\| \leq \|\tau_1^* \phi\|_{L^2}\|\tau_2^*\phi -\tau_4^*\phi\|_{L^2} \leq \|f^*(i\phi\wedge\overline\phi)\|^{1/2} \|\nu_\phi\|^{1/2} = d^{1/2} \|\nu_\phi\|^{1/2}$$
because  $f^*$ multiplies the mass of a positive measure by $d$ and 
from the proof of Proposition \ref{prop:CS-ineq},  the positive measure 
$$i(\tau_2^*\phi -\tau_4^*\phi)\wedge \overline{(\tau_2^*\phi -\tau_4^*\phi)}$$
is bounded by $\nu_\phi$. A similar estimate holds for the second term in \eqref{e:CS-branch}. The lemma follows easily.
\endproof

\subsection*{Action of $f$ on Sobolev space}
We now consider a Sobolev space of functions on $X$ and study the action of $f_*$ and $f^*$ on it. 
Let $\omega$ be a fixed Kähler form on $X$ such that $\int_X \omega=1$. For a real valued function $h$ on $X$ we define 
$$\|h\|_{W^{1,2}} := \Big| \int_X  h \, \omega \Big| + \|\partial h\|_{L^2}$$ 
and we let $W^{1,2}$ be the space of measurable functions with finite  $\|\cdot\|_{W^{1,2}}$ norm. This space and their higher dimensional counterparts were studied in \cite{dinh-sibony:decay-correlations} and \cite{vigny:dirichlet}. In the following statement and throughout the paper the symbol $\lesssim$ means an inequality up to a positive multiplicative constant.
	
	\begin{prop}\label{equinorm}
		Let $U$ be a non-empty open subset of $X$ and consider the positive measure $\nu := \mathbf{1}_U \omega$. Then the following norms on $W^{1,2}$ are equivalent to the norm $\|\cdot\|_{W^{1,2}}$.
		\vspace{-15pt}
\begin{multicols}{2}
\begin{enumerate}
        \item $\|h\|_1 := \int_X |h| \omega +\|\partial h\|_{L^2}$ \medskip
        \item $ \|h\|_2 := \|h\|_{L^2}+\|\partial h\|_{L^2}$
        \columnbreak
        \vspace*{\fill}
        \item $\|h\|_3 := |\int_X h \nu|+\|\partial h\|_{L^2}$ \medskip
        \item $\|h\|_4 := \int_X |h| \nu+\|\partial h\|_{L^2}$.
\end{enumerate}
\end{multicols}
	\end{prop}
    \begin{proof}
      By H\"{o}lder's inequality, we have $|\int_X h\omega|\leq\|h\|_{L^2}\cdot \|1\|_{L^2}$, so $\|\cdot\|_{W^{1,2}} \lesssim \| \cdot \|_2$.
      
      To show that $\|\cdot\|_2 \lesssim \| \cdot \|_3$, let $m := \int_X h\omega$ and $h_0:=h-m$.  By Poincar\'e-Sobolev's inequality, we have $\|h_0\|_{L^2} \lesssim  \|\partial h\|_{L^2}$ which implies that $\|h\|_{L^2}\lesssim \|\partial h\|_{L^2}+|m|$. On the other hand 
      $$|m|\lesssim \Big |m \nu(U) + \int_X h_0 \nu \Big | + \Big |\int_X h_0 \nu \Big | \leq  \Big|\int_X  h \nu \Big| + |\int_X |h_0| \omega 
      \lesssim \Big|\int_X  h \nu \Big| + \|\partial h\|_{L^2},$$
      where we have used  H\"{o}lder's and Poincar\'e-Sobolev's inequalities for $h_0$. It follows that  $\|h\|_{L^2} \lesssim|\int_X  h\nu|+\|\partial h\|_{L^2}$ which implies $\|\cdot\|_2 \lesssim\|\cdot\|_3$.
        
        The proof of $\| \cdot \|_3 \lesssim\|\cdot\|_4$ is trivial. We now show that  $\| \cdot \|_4 \lesssim\|\cdot\|_{W^{1,2}}$. Arguing as above and setting $U=X$ gives $\|h\|_{L^2} \lesssim|\int_X  h\omega|+\|\partial h\|_{L^2}$, so 
        $$\int_X |h| \nu\ \leq \int_X |h| \omega \leq \|h\|_{L^2} \lesssim \Big|\int_X  h\omega \Big|+\|\partial h\|_{L^2},$$ 
        giving $\| \cdot \|_4 \lesssim\|\cdot\|_{W^{1,2}}$.
        
        So far we have shown that the norms $\|\cdot\|_{W^{1,2}}, \| \cdot \|_2, \| \cdot \|_3$ and $\| \cdot \|_4$ are all equivalent. Since $\| \cdot \|_1$ is nothing but $\| \cdot \|_4$ when $U=X$, the proof is complete.
      \end{proof}
      
    \begin{prop}\label{boundop}
    	Let $f$ be a correspondence on $X$ as above. Then the operator $f^*$, acting on smooth functions, extends continuously to a bounded linear operator from $W^{1,2}$ to itself.
    \end{prop}
    \begin{proof}
    Since smooth functions are dense in $W^{1,2}$ it is enough, as in the proof of Proposition \ref{prop:CS-ineq}, to consider a smooth function $h$ and show that $\|f^*h\|_{W^{1,2}} \leq c\|h\|_{W^{1,2}}$ for some constant $c>0$ independent of $h$.
     
 By Proposition \ref{prop:CS-ineq},  we have 
 $$\|\partial(f^*h)\|_{L^2}=\|f^*(\partial h)\|_{L^2}\leq d\|\partial h\|_{L^2} \leq d \|h\|_{W^{1,2}}.$$ 
 Choose a small open subset $V$ of $X$ such that $f$ is unramified over $V$ and let $\nu:=\mathbf{1}_{V}\omega$. Then $f_*\nu$ is a positive $(1,1)$-form with bounded coefficients in $f(V)$ and hence $f_*\nu\leq c'\omega$ for some constant $c'>0$. Therefore, we have 
$$\int_X |f^*h|\nu\leq \lp f^*|h|,\nu\rp=\lp|h|,f_*\nu\rp\leq c'\lp|h|,\omega\rp=c' \int_X |h|\omega.$$ 
Now, it is enough to apply Proposition \ref{equinorm} and obtain $\|f^*h\|_{W^{1,2}} \leq c\|h\|_{W^{1,2}}$ for some constant $c>0$.
\end{proof}

\section{Non weakly modular correspondences}

In this section, we introduce the notion of weakly modular correspondences, which is motivated by Proposition \ref{prop:CS-ineq}. For non-weakly modular correspondences, we construct the canonical measures  $\mu^+$ and $\mu^-$ given by Theorem \ref{thm:main-theorem}.

    \begin{defn} \label{def:non-weakly-modular} \rm
	Let $f$ be a correspondence on $X$ such that $f$ and $f^{-1}$ have the same topological degree. We say that $f$ is {\it weakly modular} if there exists a  positive measure $m$ on $\Gamma$ and probability measures $m_1,m_2$ on $X$ such that 
	\[m=(\pi_1|_\Gamma)^*(m_1) \quad \text{and} \quad m=(\pi_2|_\Gamma)^*(m_2).\]
    \end{defn}
    
    Notice that the definition makes sense even if $X$ has dimension higher than one and in this case, we can also consider positive closed currents instead of measures. The following example justifies the choice of the name ``weakly modular'' used in the above definition.

\begin{example}[Modular correspondences] \label{ex:modular-correspondence} \rm
Let $G$ be a connected Lie group, $\Lambda \subset G$ be a torsion-free lattice and $K$ a compact Lie subgroup of $G$. The Haar measure on $G$ induces an invariant probability measure $ \lambda$ on the quotient space $X: = \Lambda\backslash G/K$. Let $g \in G$ be such that $\Lambda_g := (g^{-1} \Lambda g ) \cap \Lambda$ has finite index $d_g$ in $\Lambda$.

By passing to the quotient, the map $x \mapsto (x,gx)$ from $G$ to $G \times G$ induces a map $\Lambda_g\backslash G/K \to  X \times  X$. Denote by $\Gamma_g$ its image. An irreducible modular correspondence is a correspondence $f_g$ whose graph is  $\Gamma_g$ for some $g$ as above. We have $d_1(f_g) = d_2(f_g) = d_g$.

The two projections from $ X \times  X$ to $ X$ restricted to $\Gamma_{g}$ are non ramified and locally isometric with respect to any metric on $\Gamma_g$ induced by invariant metrics on $G$. This implies that $(\pi_1|_{\Gamma_g})^* \lambda = (\pi_2|_{\Gamma_g})^* \lambda$, so $f_g$ is weakly modular.

In general a  modular correspondence $f$ is a correspondence on $X$ whose graph is a sum  $\Gamma_f =\Gamma_{g_1} + \cdots + \Gamma_{g_m}$. It follows that $f$ is also weakly modular. For more on modular correspondences, see \cite{mok:correspondences,clozel-ullmo,dinh:modular}. In particular, modular correspondences can be characterized by the existence of some invariant measure or metric, \cite{clozel-ullmo,mok-ng}.
\end{example}

The main feature about non-weakly modular correspondences is that the operator $\frac1{d}f^*$ acting on $(1,0)$-forms is contracting. This will be crucial for the construction of the measures $\mu^+$ and $\mu^-$.

		\begin{proposition}\label{normless1}
		Let $f$ be a non-weakly modular correspondence of degree $d \geq 2$ on a compact Riemann surface $X$. Consider the operators $\frac1{d}f^*$ and $\frac{1}{d} f_*$ acting on $L^2_{(1,0)}$ which is equipped with its natural $L^2$ norm. Then we have $\|\frac1{d}f^*\|<1$ and $\|\frac1{d}f_*\|<1$.
	\end{proposition}

		We will only prove that $\|\frac1{d}f^*\|<1$, the proof of $\|\frac1{d}f_*\|<1$ being analogous. The inequality $\left \|\frac1{d}f^* \right\|\leq 1$ follows from Proposition \ref{prop:CS-ineq}. We will show that if equality holds then $f$ must be weakly modular.
		
		Suppose  that  $\left \|\frac1{d}f^* \right\|=1$. Then there exists a sequence of $(1,0)$-forms $\{\phi_n\}_{n\geq 0}$ such that $\|\phi_n\|_{L^2}=1$ and $ \left \|\frac1df^*(\phi_n) \right \|_{L^2}\to 1$. 
		By compactness, the sequence $\{i\phi_n\wedge \overline \phi_n\}_{n\geq 0}$ of probability measures admits a subsequence, which we still denote by $\{i\phi_n\wedge \overline \phi_n\}_{n\geq 0}$ for simplicity, that converges to a probability measure $m_2$. Define $m := (\pi_{2|\Gamma})^* m_2$.
		
		\begin{lemma} \label{l:test_mod}
		Let $v$ be a continuous function on $\Gamma$ such that $(\pi_{1|\Gamma})_*v=0$. Then 
		$\langle m, v \rangle = 0$.
		\end{lemma}
		
		\begin{proof}
		Without loss of generality, we can assume that $|v|\leq 1$. Define $\psi_n:=(\pi_{2|\Gamma})^*\phi_n$. 
It is enough to show that $\langle i \psi_n \wedge \overline \psi_n, v \rangle \to 0$ as $n$ tends to infinity. 
Since $ \|d^{-1}f^*(\phi_n)\|_{L^2} \to 1$, from the proof of Proposition \ref{prop:CS-ineq}, we see that
$$\nu_n:=d\, f^*(i \phi_n \wedge \overline{\phi}_n) - i f^*\phi_n \wedge \overline{f^*\phi}_n$$ 
is a sequence positive measures tending to zero. 

In order to simplify the notation, consider an arbitrary form $\phi$ in $L^2_{(1,0)}$ and define
$$\psi:=(\pi_{2|\Gamma})^*\phi \qquad \text{and} \qquad \nu_\phi:=d\, f^*(i \phi \wedge \overline{\phi}) - i f^*\phi \wedge \overline{f^*\phi}.$$
It suffices to check that $\langle i \psi \wedge \overline \psi, v \rangle =O(\|\nu_\phi\|^{1/2})$. 

Choose a connected and simply connected open set $U$ in $X\setminus B_1$ such that  $X\setminus U$ has zero Lebesgue measure (we can choose for instance $U$ such  that $X\setminus U$ is a piecewise smooth curve). So $f$ admits $d$ branches $\tau_j:U\to X$, with $1\leq j\leq d$, counting multiplicity. 
Denote by $\Gamma^{(j)}$ the graph of $\tau_j$ in $X\times X$ which is contained in $\Gamma$. Denote also by $p_j:U\to \Gamma^{(j)}$ the inverse of the map $\pi_1$ restricted to $\Gamma^{(j)}$. 
With the above choice of $U$, we can obtain from the proof of Proposition \ref{prop:CS-ineq} that 
$$\nu_\phi=\sum_{1\leq j<l\leq d} i (\tau_j^*\phi-\tau_l^*\phi)\wedge \overline{(\tau_j^*\phi-\tau_l^*\phi)}=\sum_{1\leq j<l\leq d} i (p_j^*\psi-p_l^*\psi)\wedge \overline{(p_j^*\psi-p_l^*\psi)}.$$

According to Lemma \ref{l:L2-branches}, for all indices $j,l,k,m$ we have 
$$\|i\tau_j^*\phi \wedge\overline {\tau_l^*\phi} - i\tau_k^*\phi \wedge\overline {\tau_m^*\phi}\| = O(\|\nu_\phi\|^{1/2})$$
or equivalently
$$\|ip_j^*\psi \wedge\overline {p_l^*\psi} - ip_k^*\psi \wedge\overline {p_m^*\psi}\| = O(\|\nu_\phi\|^{1/2}).$$
We deduce from the last identity and the inequality $|p_j^*v|\leq 1$ that
\begin{eqnarray*}
d^2\langle i \psi \wedge \overline \psi, v\rangle   & = &    d^2 \sum_{1\leq j\leq d} \big\langle ip_j^*\psi\wedge \overline{p_j^*\psi}, p_j^* v \big\rangle \\
& = &    \sum_{1\leq j,l,k\leq d} \big\langle ip_l^*\psi\wedge \overline{p_k^*\psi}, p_j^* v \big\rangle +O(\|\nu_\phi\|^{1/2}) \\
& = & \big\langle (\pi_{1|\Gamma})_*\psi \wedge \overline{(\pi_{1|\Gamma})_*\psi}, (\pi_{1|\Gamma})_*v\big\rangle +O(\|\nu_\phi\|^{1/2}).
\end{eqnarray*}
The first term in the last sum vanishes because by the hypothesis on $v$, we have
$$(\pi_{1|\Gamma})_*v=\sum_{1\leq j\leq d} p_j^*v=0.$$
This ends the proof of the lemma.
		\end{proof}

	\begin{proof}[End of the proof of Proposition \ref{normless1}]
	Recall that we assume by contradiction that $\left \|\frac1{d}f^* \right\|=1$ and we need to show that $f$ is weakly modular. Let $m_2, m$ be as above and define 
$m_1 := d^{-1} (\pi_1|_\Gamma)_*m$. We only need to check that  $(\pi_1|_\Gamma)^* m_1 = m $. 

Let $\chi$ be a continuous test function on $\Gamma$ and define $v := (\pi_1|_\Gamma)^*(\pi_1|_\Gamma)_* \chi  - d \chi$. Notice that the last function satisfies $(\pi_1|_\Gamma)_*v=0$. By Lemma \ref{l:test_mod}, we have  $\langle m,v\rangle=0$. This, together with the definition of $m_1$, imply that
$$\big \langle m, d\chi \big\rangle = \big\langle m , (\pi_1|_\Gamma)^*(\pi_1|_\Gamma)_* \chi \big\rangle = \big\langle (\pi_1|_\Gamma)^*(\pi_1|_\Gamma)_* m ,  \chi \big\rangle = \big \langle d (\pi_1|_\Gamma)^*m_1 ,  \chi \big \rangle.$$
The proposition follows.	
\end{proof}

\subsection*{Construction of $\mu^\pm$}

The existence of the measures $\mu^+$ and $\mu^-$ will follow from the next proposition. Once we have the contracting property of the operators $\frac{1}{d}f^*$ and $\frac{1}{d}f_*$ given by Proposition \ref{normless1}, the  proof is parallel to the $\text{dd}^c$ method used to construct the Green currents of a holomorphic endomorphism of the complex projective space $\P ^k$, see \cite{dinh-sibony:cime}.
		
	\begin{proposition} \label{expcon}
		Let $f$ be a non-weakly modular correspondence on $X$. Then, for any $h\in W^{1,2}$ the sequence $\frac1{d^n}(f^n)_*h$ converges to some constant $c^+_h$ in $W^{1,2}$. The convergence is exponentially fast in the sense that there are constants $A >0$ and $0<\lambda<1$ independent of $h$ such that 
\begin{equation} \label{ch-exp}
\Big \|\frac1{d^n}(f^n)_*h-c^+_h \Big \|_{W^{1,2}} \leq A \|\del h\|_{L^2} \lambda^n \quad \text{for every }  n \geq 0.
\end{equation}
Moreover, we have $|c^+_h|\leq A'\|h\|_{W^{1,2}}$ for some constant $A'>0$ independent of $h$ and an analogous statement holds for the operator $\frac{1}{d} f^*$.
	\end{proposition}
	
	\begin{proof}
Let 
$$c_0:=\int_X h \omega \qquad \text{and} \qquad h_0:=h-c_0.$$ 
Then, we define inductively  
$$c_n:=\int_X \Big(\frac1df_*h_{n-1}\Big)\omega \qquad \text{and} \qquad h_n:=\frac1df_*h_{n-1}-c_n.$$ 
It is not difficult to see that 
\begin{equation} \label{eq:push-forward-hn}
\frac{1}{d^n}(f^n)_* h = h_n + c_n + c_{n-1}+ \cdots + c_1+ c_0.
\end{equation}

		By Proposition \ref{boundop}, we have $h_n\in W^{1,2}$ for all $n$. Notice that $\lp \omega, h_n \rp = 0$, so by  Poincar\'e-Sobolev inequality we have $\|h_n\|_{L^2} \leq A_1 \|\partial h_n\|_{L^2}$ for some constant $A_1>0$. We also have $\del h_n = d^{-1} f_* (\del h_{n-1})$ for every $n$. Let $\lambda$ be the norm of $\frac{1}{d}f_*$ acting on $L^2_{(1,0)}$. Then $0<\lambda<1$ according to Proposition \ref{normless1} and we have  $$\|h_n\|_{L^2}\leq A_1 \left \|\partial h_n \right\|_{L^2} = A_1 \Big\|\frac 1{d^n}(f^n)_*(\partial h) \Big\|_{L^2} \leq A_1\lambda^n\|\partial h\|_{L^2}.$$
	
	By Propositions \ref{equinorm}  and \ref{boundop}, there is a constant $A_2>0$ such that $\|\frac1df_* \varphi \|_{L^2}\leq A_2\|\varphi\|_{W^{1,2}}$ for every $\varphi \in W^{1,2}$. Hence, we have
	\begin{align} \label{eq:estimate-cn}
		 		|c_n|& =\Big |\int_X \frac1d (f_*h_{n-1})\omega \Big  | \leq \Big  \|\frac1df_*h_{n-1} \Big  \|_{L^2} \leq   A_2\|h_{n-1}\|_{W^{1,2}} \nonumber \\
		& = A_2\|\partial h_{n-1}\|_{L^2} = A_2 \Big  \|\frac 1{d^{n-1}}(f^{n-1})_*(\partial h) \Big \|_{L^2} 
		\leq A_2 \lambda^{n-1}  \|\partial h\|_{L^2}. \nonumber
	\end{align}

Define  $c^+_h:= \sum^{\infty}_{k=0}c_k$. Clearly, this constant is finite and satisfies the second estimate in the proposition for a suitable constant $A'>0$.  For later use, notice that $c^+_h$ depends continuously on $h$ with respect to the weak topology on $W^{1,2}$. Indeed, when $h$ is bounded in $W^{1,2}$, the above estimates on $c_n$ are uniform on $h$ and clearly, each $c_n$ depends continuously on $h$ with respect to the weak topology on $W^{1,2}$.

For the rest of the proof we have from (\ref{eq:push-forward-hn}) that
		
		\begin{align*}
		\Big  \|\frac1{d^n}(f^n)_*h-c^+_h \Big \|_{L^2}&= \Big  \|h_n-\sum_{k=n+1}^\infty c_k \Big  \|_{L^2} \leq \|h_n\|_{L^2}+\sum_{k=n+1}^\infty |c_k| \\
		&\leq A_1 \lambda^n \|\partial h\|_{L^2}+\sum_{k=n+1}^\infty A_2 \lambda^{k-1} \|\partial h\|_{L^2} 
		\leq A_3 \| \del h\|_{L^2} \lambda^n
		\end{align*}
	for some constant $A_3 > 0$.  On the other hand, by Proposition \ref{equinorm} , we obtain
\begin{align*} 
\Big\|\frac{1}{d^n}(f^n)_*h-c^+_h \Big\|_{W^{1,2}}  & \lesssim \Big\|\frac{1}{d^n}(f^n)_*h-c^+_h \Big\|_{L^2} + \Big\|\frac{1}{d^n}(f^n)_* (\del h) \Big\|_{L^2} \\
& \leq  \Big\|\frac{1}{d^n}(f^n)_*h-c^+_h \Big\|_{L^2} +  \lambda^n \|\partial h\|_{L^2} \\
& \leq A_4 \lambda^n \|\partial h\|_{L^2}
\end{align*}
for some constant $A_4>0$. 
Thus, we get (\ref{ch-exp}) for a suitable constant $A>0$ and the proof of the proposition is now complete.
\end{proof}

	We are now in position to give a proof of our main theorem.
	
	\begin{proof}[End of proof of Theorem \ref{thm:main-theorem}]
Let $\alpha$ be a  smooth $(1,1)$-form on $X$. For simplicity, we can assume that  $c_\alpha:=\int_X \alpha = 1$. If $h$ is a smooth test function, we have 
$$\Big\lp\frac1{d^n}(f^n)^*\alpha,h \Big\rp= \Big\lp\alpha,\frac1{d^n}(f^n)_*h \Big\rp.$$

By Proposition \ref{expcon}, $\frac1{d^n}(f^n)_*h$ converges in $W^{1,2}$ to a constant $c^+_h$. Using  Proposition \ref{equinorm} we have that
$$ \Big| \Big\lp\alpha, \frac1{d^n}(f^n)_*h-c^+_h \Big\rp \Big| \lesssim \Big\| \frac1{d^n}(f^n)_*h-c^+_h \Big \|_{L^1} \lesssim \Big \| \frac1{d^n}(f^n)_*h-c^+_h \Big \|_{W^{1,2}}$$
and the last expression tends to $0$. Therefore, the sequence $\lp\alpha,\frac1{d^n}(f^n)_*h\rp$ converges to $\lp\alpha,c^+_h\rp = c^+_h$ and this limit is independent of $\alpha$. It is clear that $h \mapsto c_h^+$ is linear in $h$. The measure $\mu^+$ defined by $\lp \mu^+, h \rp := c^+_h$ satisfies (\ref{eq:def-mu-}). The relation $f^* \mu^+ = d \mu^+$ is obtained by applying $f^*$ to both sides of (\ref{eq:def-mu-}) and using the continuity of $f^*$.
		The statements about $\mu^-$ are proven analogously. This  completes the proof.
	\end{proof}
	
	\begin{proposition} \label{prop:mu-continuous}
	The functionals $h \mapsto \lp \mu^\pm, h \rp$ acting on smooth $h$ extend to linear functionals on $W^{1,2}$ which are continuous with respect to the weak topology on $W^{1,2}$. 
	\end{proposition}
	
	\begin{proof}
	We have seen in the proof of Proposition \ref{expcon} that $c^\pm_h$ depends continuously on $h$ with respect to the weak topology on $W^{1,2}$. So it is enough to define the extension of $h \mapsto \lp \mu^\pm, h \rp$  as the maps $h\mapsto c^\pm_h$.	
	\end{proof}
	
\begin{remark} \rm \label{r:integral}
Roughly speaking, Proposition \ref{prop:mu-continuous} implies that $\mu^\pm$ integrate functions in $W^{1,2}$. However, $h$ is only defined almost everywhere. So in order to give a meaning for these integrals, we need to choose a good representative of $h$ using the notion of Lebesgue points. We will not consider this matter here in order to keep the paper less technical. For $h\in W^{1,2}$ continuous, we can approximate $h$, both uniformly and in $W^{1,2}$ norm, by a sequence of smooth functions. Therefore, by continuity, $\lp \mu^\pm,h\rp$ coincides with the usual integral of $h$ with respect to $\mu^\pm$. 
\end{remark}

\begin{remark} \rm
The proof of Theorem \ref{thm:main-theorem} shows that if $\nu_n$ is a sequence of probability measures such that $|\lp \nu_n, h \rp|\leq c \|h\|_{W^{1,2}}$ for 
some constant $c>0$ and for smooth $h$ in
 $W^{1,2}$, then $d^{-n} (f^n)^* \nu_n \to \mu^+$ and $d^{-n} (f^n)_* \nu_n \to \mu^-$ exponentially fast as $n$ tends to infinity. The simple convergences require weaker estimates on $|\lp \nu_n, h \rp|$. 
\end{remark}	
	
	\begin{corollary}
		
	The measures $\mu^+$ and $\mu^-$ have no mass on polar sets.
	
	\end{corollary}

    \begin{proof}
Let $E$ be a polar set. By definition, there is a quasi-subharmonic function $u$ on $X$ such that $E\subseteq \{u=-\infty\}$. We may assume that $u \leq -1$ and $u$ is the limit of a decreasing sequence of negative smooth functions $u_n$ with $\ddc u_n\geq -\omega$. One can show that $h := - \log (-u)$ belongs to $W^{1,2}$ and is the decreasing limit of the sequence $h_n := - \log (-u_n)$ which is bounded in $W^{1,2}$, see \cite{dinh-sibony:decay-correlations} and  \cite[Example 1]{vigny:dirichlet}. 

Since $h$ is bounded from above and everywhere defined, the integral $\lp \mu^+,h\rp$ is well-defined and we have 
$$\lp \mu^+,h\rp = \lim_{n\to\infty} \lp \mu^+,h_n\rp.$$
Therefore, $\lp \mu^+,h\rp$ coincides with the value of $\mu^+$ at $h$ defined in Proposition \ref{prop:mu-continuous}, which is a finite number.
Since $h=-\infty$ on $E$, we deduce that 
 $\mu^+(E) = 0$. Analogously, we obtain that $\mu^-(E) = 0$.
   \end{proof}
 
 \subsection*{Mixing}  
 The following proposition is the equivalent in our setting of the exponential mixing for maps, see \cite{dinh-sibony:cime}.
 
 \begin{proposition} \label{prop:mixing}
 Let $f$ and $\mu^\pm$ be as in Theorem \ref{thm:main-theorem}. Then there are constants  $B > 0$ and $0< \lambda <1$ such that
 \begin{equation} \label{eq:mixing}
  \big| \lp \mu^+, d^{-n} (f^n)_* (\varphi) \cdot \psi  \rp - \lp  \mu^+, \varphi \rp \lp \mu^+,\psi \rp \big| \leq B \lambda^n \| \varphi\|_{W^{1,2}} \|\psi\|_{\mathcal C^0} 
 \end{equation}
  for every $n\geq 1$ and for $\varphi \in W^{1,2}$ and $\psi \in \mathcal C^0$. Moreover, for $0<\beta<1$  there is a constant $C = C(\beta)$ such that
   \begin{equation} \label{eq:mixing2}
  \big| \lp \mu^+, d^{-n} (f^n)_* (\varphi) \cdot \psi  \rp - \lp  \mu^+, \varphi \rp \lp \mu^+,\psi \rp \big| \leq C \lambda^{\beta n} \| \varphi\|_{\mathcal C^\beta} \|\psi\|_{\mathcal C^0}
 \end{equation} 
  for $\varphi \in \mathcal C^\beta$ and $\psi \in \mathcal C^0$. An analogous statement holds for $\mu^-$.
 \end{proposition} 
 
 \begin{proof}
Since smooth functions are dense in $W^{1,2}$ we may assume that $\varphi$ is smooth. If $\varphi$ is constant then the left hand side of (\ref{eq:mixing})  is identically zero, so the inequality trivially holds. Therefore, by adding a constant to $\varphi$, we may assume that $\varphi$ is non-constant and $\lp \mu^+, \varphi  \rp = 0$. Multiplying $\varphi$ by a constant allows us to assume that $\| \varphi\|_{W^{1,2}} =1$. In particular, $\|\del \varphi\|_{L^2} \leq 1$. Let $\varphi_n := d^{-n} (f^n)_* (\varphi)$. Since we are assuming that $c^+_\varphi = \lp \mu^+, \varphi  \rp = 0$ we have, by Proposition \ref{expcon}, that $\|\varphi_n\| _{W^{1,2}}\leq A \lambda^n$ for some constants $A>0$ and $0<\lambda<1$. Observe that $\|\, |\varphi_n|\, \| _{W^{1,2}} \lesssim \|\varphi_n\| _{W^{1,2}}$ by \cite[Prop. 4.1]{dinh-sibony:decay-correlations}. As $\mu^+$ acts continuously on $W^{1,2}$ we have that $\lp \mu^+, |\varphi_n| \rp \leq B \lambda^n$ for some constant $B > 0$.

Notice that, since  $\varphi_n$ and $\psi$ are continuous their pairings with $\mu^+$ coincide with the integral against $\mu^+$, see Remark \ref{r:integral}. We have then 
\begin{align*}
| \lp \mu^+, d^{-n} (f^n)_* (\varphi) \cdot \psi  \rp |  &= \left| \int_X \varphi_n \cdot \psi \, \diff \mu^+ \right| \leq \|\psi\|_{\mathcal C^0} \int_X |\varphi_n| \, \diff \mu^+ \\ &\leq B \|\psi\|_{\mathcal C^0} \lambda^n = B \|\psi\|_{\mathcal C^0} \|\varphi\|_{W^{1,2}} \lambda^n,
\end{align*}
which gives (\ref{eq:mixing}) because $ \lp \mu^+, \varphi  \rp = 0$.

For the second inequality we will use the interpolation theory between the Banach spaces $\mathcal C^0$ and $\mathcal C^1$, see \cite[Section 2.7.2]{triebel}. More precisely, set  $$I_n(\varphi,\psi):= \lp \mu^+, d^{-n} (f^n)_* (\varphi) \cdot \psi  \rp - \lp  \mu^+, \varphi \rp \lp \mu^+,\psi \rp .$$ Observe that $| I_n(\varphi,\psi) | $ coincides with the left hand side of (\ref{eq:mixing2}). Notice that $\| \varphi_n \|_{\mathcal C^0} \leq \| \varphi \|_{\mathcal C^0}$ for $\varphi \in \mathcal C^0$. This gives $| I_n(\varphi,\psi) | \leq 2\| \varphi \|_{\mathcal C^0} \| \psi \|_{\mathcal C^0} $, so the norm of the linear functional  $I_n(\cdot,\psi) $ acting on $\mathcal C^0$ is bounded by $ 2\| \psi \|_{\mathcal C^0}$. On the other hand, inequality (\ref{eq:mixing}) and the fact that $\|\cdot \|_{W^{1,2}} \lesssim \|\cdot \|_{\mathcal C^1}$ gives that the norm of $I_n(\cdot,\psi) $ acting on $\mathcal C^1 \subset W^{1,2}$ is bounded by $B' \lambda^n \| \psi \|_{\mathcal C^0} $ for some constant $B'$. By interpolation we get
\begin{equation} \label{eq:interpolation1}
|I_n(\varphi,\psi)| \leq C_1 \lambda^{\beta n} \| \varphi \|_{\mathcal C^\beta} \| \psi \|_{\mathcal C^0},
\end{equation}
 for $\varphi \in \mathcal C^\beta$ and $\psi
 \in \mathcal C^1$, where $C_1$ is a constant depending on $\beta$. This gives (\ref{eq:mixing2}) and completes the proof for $\mu^+$. The result for $\mu^-$ is proven in the same way.
\end{proof}  

\subsection*{Periodic points}
Let $\Gamma_{f^n}$ denote the graph of $f^n$ in $X\times X$. It may contain some copies of the diagonal but we see from the proposition below that the multiplicity of the diagonal in $\Gamma_{f^n}$ is negligible with respect to the degree $d^n$ of $f^n$. 

 \begin{proposition} \label{p:graph}
  Let $f$ and $\mu^\pm$ be as in Theorem \ref{thm:main-theorem}. Then the sequence of positive closed $(1,1)$-currents $d^{-n}[\Gamma_{f^n}]$ on $X \times X$ converges to the current $\pi_1^*(\mu^+)+ \pi_2^*(\mu^-)$ as $n$ tends to infinity. 
 \end{proposition}
 
\begin{proof}
We begin by noticing that $$\lp [\Gamma_{f^n}], \Phi \rp = \lp [X], (\pi_i|_{\Gamma_{f^n}})_* \Phi \rp = \int_X (\pi_i|_{\Gamma_{f^n}})_* \Phi, \quad i=1,2,$$ for any smooth $(1,1)$-form $\Phi$ on $X \times X$. This can be seen by considering the branches and inverse branches of $f^n$ over connected and simply connected open sets $U_i$ in $X\setminus B_i$, $i=1,2$ such that  $X\setminus U_i$ has zero Lebesgue measure.

Consider a smooth $(1,1)$-form of the type $\Phi = \pi_1^* \varphi \wedge \pi_2^* \psi$ where $\varphi$ is a smooth $(p,q)$-form and $\psi$ is a smooth $(1-p,1-q)$-form with $p,q=0$ or $1$. These forms generate a dense subspace of the space test forms of bidegree $(1,1)$ on $X \times X$ so it is enough to show that for such $\Phi$ we have
\begin{equation} \label{eq:conv-gamma_n}
\lp d^{-n} [\Gamma_{f^n}], \Phi \rp \longrightarrow \lp \pi_1^*(\mu^+)+ \pi_2^*(\mu^-), \Phi \rp \quad \text{as} \quad n \to \infty.
\end{equation}

\textit{Case 1:} $p=q=0$. In this case $\varphi$ is a smooth function and $\psi$ is a $(1,1)$-form. Let $c_{\psi} := \int_X \psi$. Notice that $(\pi_1)_* \Phi =  c_\psi \cdot \varphi$ and by degree reasons $\lp\pi_2^* (\mu^-), \Phi \rp = 0$ . Also $(\pi_1|_{\Gamma_{f^n}})_*(\Phi) = (\pi_1|_{\Gamma_{f^n}})_*(\pi_1^* \varphi \wedge \pi_2^* \psi) = \varphi \cdot (f^n)^* \psi$, so $$ \lp d^{-n}[\Gamma_{f^n}], \Phi \rp = d^{-n} \int_X  \varphi \cdot (f^n)^* \psi = \lp d^{-n} (f^n)^* \psi, \varphi \rp \to c_\psi \lp \mu^+, \varphi \rp = \lp \pi_1^*(\mu^+)+ \pi_2^*(\mu^-), \Phi \rp,$$
since $d^{-n} (f^n)^* \psi \to c_\psi \mu^+$ by Theorem \ref{thm:main-theorem}. Hence, we have (\ref{eq:conv-gamma_n}).\\

\textit{Case 2:} $p=0$ and $q=1$. Now $\varphi$ is a $(0,1)$-form and $\psi$ a $(1,0)$-form. We have $$ \left| \lp d^{-n}[\Gamma_{f^n}], \Phi \rp \right|= d^{-n} \left| \int_X  \varphi \wedge (f^n)^* \psi \right| \leq \|\varphi\|_{L^2} \|d^{-n} (f^n)^* \psi\|_{L^2} \to 0,$$
by Cauchy-Schwarz and Proposition \ref{normless1}. By degree reasons we have $\lp \pi_1^*(\mu^+)+ \pi_2^*(\mu^-), \Phi \rp = 0$ so we get (\ref{eq:conv-gamma_n}) in this case.\\

The cases $p=q=1$ and $p=1$, $q=0$ are treated similarly. This completes the proof of the proposition. 
\end{proof}

Periodic points of order $n$ of $f$ are given by the intersection $\Gamma_{f^n}\cap\Delta$. We say that a periodic point is {\it isolated} if it is given by the intersection between $\Delta$ and a component of $\Gamma_{f^n}$ which is different from $\Delta$. Isolated periodic points are counted with multiplicity. The above result suggests that there are about $d^n+o(d^n)$ repelling (resp. attracting) isolated periodic points which are equidistributed with respect to $\mu^+$ (resp. $\mu^-$). Indeed, $\pi_1^*(\mu^+)$ (resp.\ $\pi_2^*(\mu^-)$) can be imagined as limits of graphs with large (resp.\ small) slopes and its intersection with $\Delta$ can be identified with $\mu^+$ (resp.\ $\mu^-$).

\section{Equidistribution of pre-images}

This section is devoted to the proof of Theorem \ref{thm:equid-no-periodic-crit}. We will only prove (\ref{eq:equid-no-periodic-crit}) since the proof of (\ref{eq:equid-no-periodic-crit2}) is completely analogous. Instead of working with $\frac{1}{d}f^*$ acting on measures, by duality, we will work with the operator $\Lambda := \frac{1}{d}f_*$ acting on the space $W^{1,2}$. By the theory of interpolation of Banach spaces, if we prove (\ref{eq:equid-no-periodic-crit}) for $\beta = 1$, the estimate for $0< \beta \leq 1$ will follow, see \cite{triebel} for details. So we can assume that $\beta=1$ and $\varphi \in \mathcal C^1$. 

After normalizing  $\varphi$ to make it satisfy $\langle \mu^+, \varphi \rangle = 0$ and $\|\varphi\|_{\mathcal C^1} \leq 1$, we need to show that $\Lambda^n \varphi$ converges to 0 exponentially fast as $n$ tends to infinity. This will be achieved by combining an exponential integral estimate for functions in $W^{1,2}$ and a Lojasiewicz's inequality.

The following estimate follows from the classical Moser-Trudinger theorem, see e.g. \cite{moser:trudinger}. It replaces the usual exponential estimate given by Skoda's theorem \cite{skoda} which has been used earlier in some equidistribution problems for maps, where quasi-p.s.h. functions are considered as test functions, see \cite{dinh-sibony:cime}.

\begin{proposition} \label{prop:exponential-estimate}
Let $\mathcal F$ be a bounded family in $W^{1,2}$. Then there are constants $A_{\mathcal F}>0$ and $\theta_{\mathcal F} > 0$, depending on ${\mathcal F}$, such that 
$$\int_X e^{\theta_{\mathcal F} \, \varphi^2} \omega  \leq A_{\mathcal F} \quad \text{for every} \quad \varphi \in \mathcal F.$$
\end{proposition}

We now make some elementary observations. Given $x \in X$, a branch of order $n$ starting at $x$ is a sequence of the form 
$$\mathbf x = \left(x_0, x_1,x_2,\ldots,x_n \right) \text{ with } x_0=x \text{ and } x_{j} \in f(x_{j-1}) \text{ for } j=1,\ldots,n-1.$$ 
Here, the points, and hence the branches, are repeated according to their multiplicities. Every point $x$ admits $d^n$ branches of order $n$ counted with multiplicity.

Recall that $B_2$ denotes the set of critical values of $f$ introduced in Section \ref{sec:preliminaries}. This is the largest finite subset of $X$ such that $\pi_2$ restricted to each germ of $\Gamma$ outside $\pi_2^{-1}(B_2)$ is unramified.
We have the following elementary fact, which says that the number of times a branch of $f$ visits $B_2$ is uniformly bounded.

\begin{lemma} \label{lemma:finite-visits}
Let $f$ be as in Theorem \ref{thm:equid-no-periodic-crit} and  let $\mathbf x = (x_0,x_{1},\ldots,x_{n})$ be any branch of order $n$ of $f$. Then  we have $\#\{j : x_{j} \in B_2 \} \leq \# B_2$.
\end{lemma}
\begin{proof}
If the result is not true then there are $j_1$ and $j_2$ , with $j_1 < j_2$, such that $x_{j_1}, x_{j_2}$ are equal and belong to $B_2$. By setting $z:= x_{j_1}=x_{j_2}$ and $m:= j_2-j_1$, we have that $(z,x_{j_{1}+1},\ldots,x_{j_{2}-1},z)$ is a branch of order $m$ of $f$. This contradicts the hypothesis that  no critical value of $f$ is periodic.
\end{proof}

Let $\Gamma^{(n)}$ be the graph of $f^n$ in $X\times X$. An irreducible germ $V$ of $\Gamma^{(n)}$ at a point $(x,y)$ can be seen as the graph of a local multi-valued holomorphic map sending $x$ to $y$. We will call this map {\it a germ of $f^n$ at the point $x$}. Its {\it degree} is, by definition,  the local degree of 
$\pi_2|_V$ at the point $(x,y)$. 
Denote by $\kappa_n$ the maximal degree of a germ of $f^n$.

\begin{lemma} \label{l:small-mult}
The integer $\delta:=\sup_{n\geq 0} \kappa_n$ is finite.
\end{lemma}
\proof
Consider a germ of $f^n$ as above that we denote by $g$. Assume that its degree is maximal, i.e. equal to $\kappa_n$, at $(x,y)$. Observe that there is a branch  $(x_0,\ldots, x_n)$,  with $x_0=x$, $x_n=y$, and local germs $g_j$ of $f$ sending $x_{j-1}$ to $x_j$ such that 
$g$ is a component of $g_n\circ \cdots \circ g_1$. More precisely, although the graphs of $g_j$ are irreducible, the graph of their composition may be reducible and contains the graph of $g$ as a component. 

By Lemma \ref{lemma:finite-visits}, every $g_j$, except for at most $\#B_2$ of them, are unramified, i.e.\ , of topological degree $1$.  For those which are ramified, the degree is at most equal to the topological degree $d$ of $f$. It follows that the degree of $g_n\circ \cdots \circ g_1$ is at most equal to $d^{\# B_2}$. Thus, the number $\delta$ is bounded above by $d^{\# B_2}$.
\endproof

Let $\lambda_0$ be the norm of $\frac{1}{d}f_*$ acting on $L^2_{(1,0)}$. From Theorem \ref{normless1} applied to $f^{-1}$, we see that $\lambda_0 < 1$. Choose a number $\lambda$ such that $\lambda_0 < \lambda < 1$. Let $\delta$ be as in Lemma \ref{l:small-mult} and fix an integer $N$ large enough so that $\delta < \lambda^{-N}$. By replacing $f$ and $\lambda$ by $f^N$ and $\lambda^N$,  
we may assume that $\delta < 1/\lambda$.

We recall the following version of Lojasiewicz's inequality. Since here  $X$ is a Riemann surface, the estimate can be easily proved using the  fact that any local holomorphic map from $X$ to $X$ is given, in suitable coordinates centered at 0, by $z \mapsto z^m$ for some integer $m \geq 1$. 

\begin{lemma} \label{l:lojasiewicz}
Let $f$ be a holomorphic correspondence on $X$ such that both $f$ and $f^{-1}$ have degree $d$. Assume that the local degree of every germ of $f$ is smaller than or equal to $\delta$. Then there is a constant $M \geq 1$ such that for any $x,y \in X$ we can write $f^{-1}(x)= \{x_1,\ldots,x_d\}$ and $f^{-1}(y)= \{y_1,\ldots,y_d\}$ with $\dist (x_j,y_j) \leq M \dist(x,y)^{1 \slash \delta}$ for every $j=1,\ldots,d$.
\end{lemma}

Consider now the operator $$\Lambda := \frac{1}{d} f_*$$ acting on $W^{1,2}$. Let $\varphi$ be a $\mathcal{C}^1$ function on $X$ with $\|\varphi\|_{\mathcal C^1} \leq 1$. We have $\varphi \in W^{1,2}$ and $\varphi$ is Lipschitz continuous with Lipschitz constant equal to $1$. Define $\varphi_n := \Lambda^n \varphi$ for $n\geq 0$.

In what follows, we say that a real valued function $u$ is {\it $(M,\gamma)$-H\"older continuous} if $|u(x) - u(y)|  \leq M \dist(x,y)^\gamma$ for every $x,y \in X$.

\begin{proposition} \label{prop:holder-loja}
Let $\varphi$ and $\varphi_n$ be as above. Then there is a constant $M \geq 1$ such that $\varphi_n$ is $(M^n,\delta^{-n})$-H\"older continuous for every $n\geq 0$.
\end{proposition}

\begin{proof}
Since we are assuming that every germ of $f$ has degree $\delta$ or less, by Lemma \ref{l:lojasiewicz},  there is a constant $M \geq 1$ such that for any $x,y \in X$ we can write $f^{-1}(x) = \{x_1, \ldots, x_d\}$ and $f^{-1}(y) = \{y_1, \ldots, y_d\}$ with $\dist (x_j,y_j) \leq M \, \dist(x,y)^{1 \slash \delta}$ for every $j$.

We will prove the result by induction. For $n=0$ the result is obvious, so we assume it is true for $n$. Then, we have
\begin{align*}
\left| \varphi_{n+1}(x) - \varphi_{n+1}(y) \right| &= \frac{1}{d} \left| f_*\varphi_{n}(x) - f_*\varphi_{n}(y) \right| \leq \frac{1}{d} \sum_{i=1}^d \left| \varphi_{n}(x_i) - \varphi_{n}(y_i) \right| \\
&\leq M^n \, \dist(x_i,y_i)^{1 \slash \delta^n} \leq M^{n+1}\, \dist(x,y)^{1 \slash \delta^{n+1}}, 
\end{align*}
which gives what we wanted.
\end{proof}

Let 
$$\mathcal W^+:=\big\{\varphi \in W^{1,2} : \langle \mu^+,\varphi \rangle = 0,  \text{ and } \|\del \varphi \|_{L^2} \leq 1 \big\}.$$
Here, the pairing $\langle \mu^+,\varphi \rangle$ is defined by Proposition \ref{prop:mu-continuous}.

\begin{lemma} \label{lemma:sobolev-norm-bounded-W+}
The set $\mathcal W^+$ is bounded in $W^{1,2}$ and invariant by $\Lambda$.
\end{lemma}

\begin{proof}
We first show that $\mathcal W^+$ is bounded in $W^{1,2}$. 
Let $\varphi \in \mathcal W^+$. Since $\|\del \varphi \|_{L^2} \leq 1$  we only need to prove that $\|\varphi\|_{L^2}$ is bounded by a constant independent of $\varphi$. 
By Poincar\'e-Sobolev inequality, we have
$$\|\varphi\|_{L^2} \leq |m(\varphi)| + c \|\del \varphi \|_{L^2},$$ 
where $m(\varphi) := \lp \omega, \varphi \rp$ and $c>0$ is a constant.

Consider the function $\widetilde \varphi := \varphi - m(\varphi)$. We have $m(\widetilde \varphi) = 0$ and $\|\del \widetilde \varphi \|_{L^2} = \|\del \varphi \|_{L^2} \leq 1$. So by the above Poincar\'e-Sobolev inequality applied to $\widetilde\varphi$, we see that $\|\widetilde \varphi\|_{W^{1,2}}$ is  bounded independently of $\varphi$. 
Moreover, we have 
$$|m(\varphi)| = |  \langle \mu^+, \varphi -  \widetilde \varphi \rangle| = |\langle  \mu^+,  \widetilde \varphi \rangle|$$
and by Proposition \ref{prop:mu-continuous}, the last expression is bounded by a constant independent of $\varphi$. Thus, $\|\varphi\|_{L^2}$ is bounded independently of $\varphi$. 

It remains to prove that $W^{1,2}$ is invariant by $\Lambda$ or equivalently to show that $\Lambda\varphi$ belongs to $\mathcal W^+$. 
When $\varphi$ is smooth, $\Lambda\varphi$ is continuous. By Remark \ref{r:integral}, the pairing $\lp\mu^+,\Lambda\varphi\rp$ is just the usual integral and the desired property follows from the $f^*$-invariance of $\mu^+$ and Proposition \ref{prop:CS-ineq}. 

It is not difficult to see that smooth functions are dense in $\mathcal W^+$. Therefore, the case of non-smooth $\varphi$ follows from the continuity of the operator $\Lambda$ and the pairing $\lp\mu^+,\cdot\rp$ on $W^{1,2}$. This ends the proof of the lemma.
\end{proof}

\begin{proposition} \label{prop:holder-exponential}
There is a constant $A > 0$ (independent of $M$ and $\gamma$) such that if $\varphi \in \mathcal W^+$ is $(M,\gamma)$-H\"older continuous for some constants $M \geq 1$ and $0 < \gamma \leq 1$, then $$\|\varphi\|_\infty \leq A \gamma^{-1}(1+\log M).$$
\end{proposition}
\begin{proof}
Suppose the conclusion is false. Then, for every constant $A>1$, we can find a point $a$ such that $|\varphi(a)| >  A \gamma^{-1}(1+\log M)$. 
Fix a constant $A$ large enough. We claim that there is a constant $0<r<1/2$ such that $|\varphi| \geq \frac{1}{2} A\log(1 \slash r)$ on the disc $\DD(a,r)$ of radius $r$ centered at $a$. 

In order to prove the claim, let $r := \frac{1}{2} M^{-1 \slash \gamma}$ and consider an arbitrary point $b$ such that $\dist(a,b) \leq r$. Then, we have\begin{align*}
|\varphi(b)| &\geq - |\varphi(b) - \varphi(a)| + |\varphi(a)| \geq - M \dist(b,a)^\gamma +  A \gamma^{-1}(1+\log M) \\
&\geq -1 +  A \gamma^{-1}(1+\log M) \geq \frac{1}{2} A\log(1 \slash r),
\end{align*}
which proves the claim.

Now, let $\theta:=\theta_{\mathcal F}$ be the constant in Proposition \ref{prop:exponential-estimate} for ${\mathcal F}:= {\mathcal W}^+$. 
Then, we have 
$$\int_{\DD(a,r)} e^{\theta |\varphi|^2}\omega \geq \int_{\DD(a,r)} e^{\theta |\varphi |} \omega \gtrsim r^{2 - \theta A\slash 2 }.$$ 
We get a contradiction with Proposition \ref{prop:exponential-estimate} because $r<1/2$ and $A$ can be taken arbitrarily large. This ends the proof of the proposition.
\end{proof}

\begin{proof}[End of the proof of Theorem \ref{thm:equid-no-periodic-crit}] 
Recall that we only need to prove the estimate (\ref{eq:equid-no-periodic-crit}) for $\beta=1$. Let $\lambda,\lambda_0$ and $\delta$ be as before. As seen above we may assume that $\delta < 1 \slash \lambda$.

Let $\varphi$ be a function of class $\mathcal C^1$. If $\varphi$ is constant then it is clear that (\ref{eq:equid-no-periodic-crit}) holds. Therefore, we may subtract a constant from $\varphi$ and assume that $\langle \mu^+,\varphi \rangle = 0$. We may also multiply $\varphi$ by a non-zero constant and assume that  $\|\varphi\|_{\mathcal C^1} \leq 1$. Letting $\varphi_n := \Lambda^n(\varphi)$ as before, we have  
$$\langle d^{-n}(f^n)^*\delta_a, \varphi \rangle = \langle \delta_a, \varphi_n \rangle = \varphi_n(a).$$
So we need to prove that $|\varphi_n(a)| \leq A \lambda_+^{-n}$ for some constants $A>0$ and $0<\lambda_+<1$.

From Proposition \ref{prop:holder-loja} there is a constant $M \geq 1$ such that  $\varphi_n$ is $(M^n,\delta^{-n})$-H\"older continuous for every $n \geq 0$. Define $\widehat \varphi_n := \lambda^{-n} \varphi_n$. This is a $(\lambda^{-n} M^n,\delta^{-n})$-H\"older continuous function. We also have 
$\langle \mu^+, \widehat \varphi_n \rangle = 0$  and
$$\|\del \widehat \varphi_n\|_{L^2} = \lambda^{-n} \|\del \Lambda^n(\varphi)\|_{L^2} =  \lambda^{-n} \|d^{-n}(f^n)^*(\del \varphi)\|_{L^2} \leq \lambda^{-n} \lambda_0^n \|\del \varphi\|_{L^2}.$$
Since $\lambda_0 < \lambda$, the function $\widehat \varphi_n$ belongs to $\mathcal W^+$. By Proposition \ref{prop:holder-exponential} we have 
$$\|\widehat \varphi_n\|_{\infty} \leq A\delta^n\big(1 + \log (\lambda^{-n} M^n)\big) \leq An\delta^n \big(1 + \log (\lambda^{-1} M)\big) $$
or equivalently
$$\|\varphi_n\|_{\infty} \leq A\delta^n \lambda^n \big(1 + \log (\lambda^{-n} M^n)\big) \leq An\delta^n \lambda^n\big(1 + \log (\lambda^{-1} M)\big).$$ 
 
Finally, since $\delta < 1/\lambda$, by choosing $\lambda_+$ such that $\delta\lambda<\lambda_+<1$ and $A^+$ a constant  large enough, we obtain that $\| \varphi_n\|_{\infty}\leq A^+\lambda_+^n$. This proves the estimate (\ref{eq:equid-no-periodic-crit}) for $\beta=1$. The theorem follows.
\end{proof}


\bibliographystyle{alpha}

\bibliography{refs}

\end{document}